\documentclass[12pt,oneside,reqno]{amsart}

\usepackage[letterpaper,textwidth=6.9in,textheight=9.6in,centering]{geometry}
\usepackage{amsmath, amsfonts, amssymb}
\usepackage{mathrsfs}
\usepackage{latexsym}
\usepackage{bbm}
\usepackage{mathtools}
\usepackage{exscale}
\usepackage{cases}

\usepackage{enumitem}

\usepackage[pdfstartview=FitH,
            CJKbookmarks=true,
            bookmarksnumbered=true,
            bookmarksopen=true,
            colorlinks=true,
            linkcolor=blue,
            anchorcolor=blue,
            citecolor=red,
            urlcolor=blue
            ]{hyperref}

\numberwithin{equation}{section}
\newtheorem{theorem}{Theorem}[section]
\newtheorem{definition}[theorem]{Definition}
\newtheorem{lemma}[theorem]{Lemma}
\newtheorem{corollary}[theorem]{Corollary}

\newtheorem{proposition}[theorem]{Proposition}

\newtheorem{definition and theorem}[theorem]{Definition and Theorem}

\def\bl{\begin{lemma}}
\def\el{\end{lemma}}
\def\bc{\begin{corollary}}
\def\ec{\end{corollary}}
\def\bt{\begin{theorem}}
\def\et{\end{theorem}}

\def\bp{\begin{proposition}}
\def\ep{\end{proposition}}
\def\be{\begin{equation}}
\def\ee{\end{equation}}
\def\baa{\begin{align*}}
\def\eaa{\end{align*}}

\theoremstyle{definition}
 
\theoremstyle{remark}
\newtheorem{remark}[theorem]{Remark}

\DeclareMathOperator{\conv}{conv}
\DeclareMathOperator{\aff}{aff}
\DeclareMathOperator{\lin}{lin}

\newcommand{\R}{\mathbb{R}}

\newcommand{\VP}{\mathcal P}

\newcommand{\I}{\mathcal I}

\newcommand{\F}{\mathcal F}

\newcommand{\A}{\mathcal A}
\newcommand{\sn}{\mathbb S^2}
\newcommand{\intt}{\operatorname{int}}

\newcommand{\cP}{\mathcal{P}}

\newcommand{\BM}{\mathrm{BM}}
\newcommand{\Aff}{\operatorname{Aff}}

\newcommand{\p}{{\mkern1.37mu }}

\title[The Mahler Conjecture in  Three Dimensions]{\large The Mahler Conjecture in Three Dimensions}

\author{Shibing Chen}
\address{School of Mathematical Sciences,
University of Science and Technology of China,
Hefei, Anhui 230026, China}
\email{chenshib@ustc.edu.cn}

\author{Yuanyuan Li}
\address{Institute for Theoretical Sciences,
Westlake University, Hangzhou, 310030, China}
\email{lyyuan@westlake.edu.cn}

\author{Dongmeng Xi}
\address{Department of Mathematics,
Shanghai University,
Shanghai 200444, China}
\email{xi\_dongmeng@shu.edu.cn; dongmeng.xi.math@gmail.com}

\author{Zhe-Feng Xu}
\address{School of Mathematical Sciences,
University of Science and Technology of China,
Hefei, Anhui 230026, China}
\email{xzf1998@mail.ustc.edu.cn}
\date{\today}

\begin{document}

\begin{abstract}
The Mahler conjecture dates back to 1938. This paper solves the conjecture for general convex bodies in three dimensions by developing a  method called the shadow flow. The equality case is   characterized as well.  This method is also applied to give  a new proof of the  three-dimensional symmetric case, which was first proved  by Iriyeh--Shibata. 
\end{abstract}



\maketitle

\medskip

\section{Introduction}
Let $K\subset \R^n$ be a convex body containing the origin $o$ in its interior. The {\it polar body} of $K$ is defined by
\[ K^\circ=\{y\in\R^n: y\cdot x\leq 1\text{ for all }x\in K\}. \]
It is a basic duality that $(K^\circ)^\circ=K.$

The Blaschke-Santal\'o inequality, a classical affine
isoperimetric inequality dating back to the work of Blaschke \cite{B1917} in 1917 and Santal\'o \cite{Santalo1949} in 1949, can be stated as follows. If $K\subset \R^n$ is a convex body whose centroid is at the origin, then 
\[  |K|\, |K^\circ| \le \omega_n^2,\]
with equality if and only if $K$ is an ellipsoid. 
Here $\omega_n$ denotes the volume of the unit ball in $\R^n$. See the survey of Lutwak \cite{Lutwaksurvey} and the $L_p$-Blaschke-Santal\'o inequality by Lutwak--Zhang \cite{LutwakZhang1997}. 
A recent breakthrough is that B\"or\"oczky--Patsalos--Saroglou \cite{BPS-2026} proved Ball's conjecture, a strengthened Blaschke-Santal\'o inequality.

It is a longstanding problem to establish the sharp reverse Blaschke-Santal\'o inequality, originating in Mahler's works \cite{Mahler,Mahler1939} in 1938 and 1939. This problem is precisely the Mahler conjecture. 

\vskip 3pt
\noindent{\bf Mahler conjecture.} {\it If  $K \subset \R^n$ is a convex body  with $o\in \intt K$, then   
\[ |K|\, |K^\circ| \ge \frac{(n+1)^{n+1}}{(n!)^2},\]
with the sharp constant attained by simplices with centroid  at the origin.

When $K $ is origin-symmetric,   
\[ |K|\, |K^\circ|  \ge \frac{4^{n}}{n!},\]  
with the sharp constant attained by cubes.} 

\vskip 3pt  

In two dimensions, Mahler \cite{Mahler} originally established the above two inequalities, in both the general and the symmetric forms. In higher dimensions, the Mahler conjecture had remained open for more than seven decades. In 2020, Iriyeh--Shibata \cite{IS2020} solved the symmetric form in three dimensions, and fully characterized the equality cases. 
The general form was independently formulated as a conjecture,  notably by A. D. Aleksandrov \cite{ADAlex} in 1967.  

Let us recall some known progress on the Mahler conjecture and related studies. Based on Iriyeh--Shibata's proof of the three-dimensional symmetric case \cite{IS2020}, a streamlined proof was subsequently given by Fradelizi, Hubard, Meyer, Rold\'an-Pensado, and Zvavitch \cite{FHMRZ}.
Both proofs rely on ingenious algebraic-topological equipartition arguments. 
Iriyeh--Shibata \cite{IS2025} also obtained some partial results in the three-dimensional general form. In 1987, Bourgain--Milman \cite{BourgainMilman1987} established an asymptotic reverse Santal\'o inequality; Kuperberg \cite{Kuperberg2008} improved the constant by a different method, and Berndtsson \cite{Bnsn}
gave a streamlined complex-analytic proof of Kuperberg's improved estimate. The equality cases in the planar Mahler conjecture were characterized by Reisner \cite{Reisner1986} in the symmetric setting and by Meyer \cite{Meyer1991} in the general setting.
Reisner \cite{Reisner1986} also proved the zonoid case, and Gordon--Meyer--Reisner \cite{GMR1988} gave a simple new proof. Saint-Raymond \cite{SaintRaymond1981} proved the unconditional case, and Meyer \cite{Meyer1986} gave an influential new proof;  Barthe--Fradelizi \cite{Barfra2013} obtained further results for bodies with many hyperplane symmetries.  
See the works of Nazarov--Petrov--Ryabogin--Zvavitch \cite{NPRZ2010}, Kim--Reisner \cite{KimReisner2011}, and Kim \cite{Kim2014} for related strict local minimality results in the Banach--Mazur distance; see also  Reisner--Sch\"utt--Werner 
\cite{RSW2012} for curvature restrictions on possible minimizers. The Mahler conjecture  also has important connections with symplectic capacities and the slicing problem, due respectively to Artstein-Avidan, Karasev, and Ostrover \cite{MR3263026}, and Klartag \cite{klartag2018}.
For further background, see the survey \cite{FMZsurvey}. Related Mahler-type sharp affine inequalities with simplices as extremizers were obtained by Lutwak--Yang--Zhang \cite{LYZplms,LYZ2010}. Stability results for the volume product and related affine inequalities were obtained by B\"or\"oczky \cite{B2010}, B\"or\"oczky--Hug \cite{BH2010}, and B\"or\"oczky--Makai--Meyer--Reisner \cite{BMMR2013}. 

In 2006, Campi--Gronchi \cite{CG} and Meyer--Reisner \cite{MR} brought the method of shadow systems 
into the study of the Mahler conjecture. In the centrally symmetric setting, Campi--Gronchi proved a polar-convexity theorem along shadow systems and used it to give a new proof of the two-dimensional symmetric form. Meyer--Reisner established  a general  polar-convexity theorem together with a rigidity result, and applied them to prove the general form for polytopes in \(\R^n\) with at most \(n+3\) vertices. Fradelizi--Meyer--Zvavitch \cite{FMZ2012} further used this approach to prove the general form for three-dimensional bodies that are convex hulls of two planar convex bodies.

The present paper establishes the three-dimensional general form of the Mahler conjecture and also provides a new proof of the three-dimensional symmetric form. Our main result for the general form of the Mahler conjecture is the following.

\vskip 5pt
\noindent {\bf Main Result.} {\it Every convex body $K\subset\R^3$ with $o\in \intt K$ satisfies
\[
      |K|\,|K^\circ| \geq \frac{64}{9},
\]
with equality if and only if $K$ is a tetrahedron (three-dimensional simplex) whose centroid is at the origin.}
\vskip 5pt

We now discuss the deformation viewpoint behind the proof. A classical starting point for deforming convex bodies is Steiner symmetrization, introduced by Steiner \cite{Steiner1838} in 1838  and presented systematically by Blaschke \cite{Blaschke1916} in 1916. Steiner symmetrization has proved to be a powerful deformation method for proving affine isoperimetric inequalities whose extremizers are ellipsoids; see, e.g., the solution by E. Milman and Yehudayoff \cite{MilmanYehudayoff2023} of a conjecture of Lutwak, as well as other works on affine isoperimetric inequalities by Lutwak--Zhang, Lutwak--Yang--Zhang, and Xi--Zhao 
\cite{LutwakZhang1997,LYZ2000,XiZhao2022}. For symmetrization methods and related applications in convex and affine geometry, see Bianchi--Gardner--Gronchi \cite{BGG2017}, Haddad--Ludwig \cite{HLJDG,HLJEMS}, and Saroglou \cite{Saroglou2022}. 
However, \emph{reverse affine isoperimetric inequalities} like the Mahler conjecture, whose extremizers are the most singular bodies, are much more challenging; see, e.g., the landmark results of Ball's volume ratio inequalities \cite{Ball1989,Ball1991} and Zhang's projection inequality \cite{Zhang1991}; see also Gardner--Zhang \cite{GZ1998} for a new proof and a generalization of Zhang's inequality. 
A central challenge in studying these reverse inequalities is the lack of a general deformation method toward their expected singular extremizers, such as simplices, cubes, octahedra, or, more generally, Hanner polytopes.

Along this reverse direction, there are some known methods in two dimensions. Blaschke's 1917 work \cite{Blaschke1917Sylvester} on Sylvester's four-point problem used a deformation known as ``shaking down''; see Campi--Colesanti--Gronchi \cite{CCG1999} for a general version of Sylvester's problem. 
Blaschke's approach was used to prove a reverse planar random-simplex inequality whose singular extremizer is a triangle, while the smooth extremizer is an ellipse. This is an early example of a deformation method leading toward a singular body in the plane. Mahler's planar proof \cite{Mahler} in 1938 also belongs to this deformation direction.

A shadow system is a general deformation that generalizes the classical Steiner symmetrization. Its original name, introduced by Rogers--Shephard \cite{RogersShephard1958} in 1958, was ``linear parameter system''. Shephard \cite{Shephard1964} later used the name ``shadow system''. A key idea of our work is to introduce shadow systems with special structures that can deform arbitrary polytopes and allow us to identify the bodies with the simplest possible structure in $\R^3$: simplices in the general case, and affine images of the cube or the octahedron in the symmetric case. These special structures are called \emph{admissible speeds}. We call such a shadow system with admissible speeds a \emph{shadow flow}.

Basic properties of the shadow flow are established in Section \ref{sec:admissible} and  developed in Section \ref{sec:lem:local-exclusion}, see also  Remark \ref{rem:flow} for further discussion. 
Section \ref{sec:main6.1}, together with the counting argument in Section \ref{sec:combinatorial}, completes the proof of the Mahler conjecture for general bodies; Section \ref{sec:main6.2} characterizes the equality case.
Section \ref{sec:final} applies the shadow-flow method to give a new proof of the symmetric form.

\section{Basic notation and tools}\label{Sec:tools}

Throughout, $B^n$ denotes the unit ball in $\R^n$. By a polytope, we always mean a convex polytope with non-empty interiors. By $\aff A$, $\lin A$, $\conv A$, we mean the affine hull, linear span, and convex hull of a set $A\subset \R^n$. The dimension of a set is defined as the dimension of its affine hull. We shall also use the standard notions of vertices (or extreme points), faces, and facets of a convex set; see Schneider \cite{Schneider} and Gardner \cite{Gardner} for  
general references. 

For a convex body $K \subset \R^n$ and a point $z \in \operatorname{int} K$, define the {\it polar body} with center $z$ by
\[
    K^z=\{y\in\R^n:(y-z)\cdot(x-z)\leq 1\text{ for all }x\in K\}.
\]
Equivalently, $K^z=(K-z)^\circ+z$, where $(\cdot)^\circ$ is the usual polar with respect to the origin. The bipolar identity gives
\[
        (K^z)^z=K.
\]

Let $s(K)$ be the {\it  Santal\'o point} of $K$, namely the unique point of $\operatorname{int}K$ minimizing $z\mapsto |K^z|$; see \cite{Santalo1949}. In fact,    $z\mapsto |K^z|$ is convex on $\intt K$, and $|K^z| \to \infty$ as $z\to \partial K$. 
The  {\it volume product} (or {\it Mahler volume}) is defined by
\[
        \VP(K)=|K|\,|K^{s(K)}|,
\] 
where $K^{s(K)}$ is called the {\it Santal\'o polar} of $K$. The {\it centroid} of $K$ is given by
\[
c(K) :=\frac{1}{|K|}\int_K x \,dx.
\]
With our convention, the duality between the centroid and the Santal\'o point reads (see \cite{Santalo1949})
\begin{equation}\label{centr-santalo} c(K^{s(K)})=s(K)  \qquad {\rm and} \qquad c(K)=s(K^{c(K)}). \end{equation}
$\VP$ is invariant under affine isomorphisms; see, for example, \cite[Section 1]{FMZ2012}. It is well known that the Santal\'o point is dual to the centroid in the sense that $o$ is the Santal\'o point of $K$ if and only if the centroid of $K^\circ$ is $o$. Moreover,  the Santal\'o point of a simplex coincides with its 
centroid.
Therefore, the simplex value in three dimensions is 
\[
        \VP(\Delta_3)=\frac{(3+1)^{3+1}}{(3!)^2}=\frac{64}{9}.
\]
We record some basic properties of volume product and polar bodies.
\begin{lemma}\label{lem:polar-nonincreasing}
For every convex body $K\subset\R^n$,
\[
        \VP(K^{s(K)})\leq \VP(K).
\]
\end{lemma}
Indeed, if we denote $ K^*=K^{s(K)}$, then
$$\VP(K^{s(K)})=|K^*||(K^*)^{s(K^*)}|\leq |K^*||(K^{s(K)})^{s(K)}|=|K||K^*|=\VP(K).$$
 
 \begin{lemma}\cite[Theorem 1.3]{BayerLee1993}\label{lem:polarity-face-lattice}
Let $P\subset\R^n$ be a full-dimensional polytope and let $z\in\operatorname{int}P$. Then $P^z$ is a full-dimensional polytope,  and
\[
        V(P^z)=F(P),\qquad F(P^z)=V(P).
\]
\end{lemma}

The following lemma follows from the affine invariance property of volume product, and the basic fact that the
number of vertices of the limit polytope is no more than the limsup of the number of vertices of the convergent sequence of polytopes.
\begin{lemma}[Compactness in a bounded vertex class]\label{lem:bounded-class}
Fix $N\geq n+1$. Let $\mathcal C_N$ be the class of all  polytopes in $\R^n$ with at most $N$ vertices. Then $\VP$ attains its infimum on $\mathcal C_N$.
\end{lemma}

\begin{proof}
Let $(P_m)\subset\mathcal C_N$ be a minimizing sequence. By affine invariance of $\VP$, we may assume
 each $P_m$ is in John position after performing an affine transformation. Thus, by John's lemma  we have
\[
        B^n\subset P_m\subset nB^n.
\]
Write
\[
        P_m=\conv\{x_{m,1},\ldots,x_{m,N}\},
\]
allowing repetitions when $P_m$ has fewer than $N$ vertices. Passing to a subsequence we may assume $x_{m,i}\to x_i$ for every $i$. Set
\[
        P=\conv\{x_1,\ldots,x_N\}.
\]
The inclusions pass to the Hausdorff limit, so $B^n\subset P$ and $P$ is full-dimensional. Also $P\in\mathcal C_N$. Hausdorff-continuity of the volume product \cite[Lemma 3]{FMZ2012} then gives
\[
        \VP(P)=\lim_{m\to\infty}\VP(P_m),
\]
so $P$ is a minimizer.
\end{proof}

A \emph{shadow system} along a unit direction \(\theta\in\mathbb S^{n-1}\), as introduced  by Rogers and Shephard \cite{RogersShephard1958}, is a family 
\[
        L_t=\conv\{x+t\alpha(x)\theta:x\in B\},
\]
where the {\it base set}  $B\subset\R^n$ is bounded, the function $\alpha:B\to\R$ is bounded, and $t$ ranges over an interval. It is non-degenerate if every $L_t$ has non-empty interior. The following result of Meyer and Reisner, which combines Theorem 1, Proposition 7 and their consequences in \cite{MR}, is very important for our proof. One can also refer to the complete statement in \cite[Proposition 1]{FMZ2012}.

\begin{theorem}[Theorem of Meyer--Reisner \cite{MR}] \label{thm:shadow}
Let $(L_t)_{t\in[-a,a]}$ be a non-degenerate shadow system in $\R^n$ along a direction $\theta\in \mathbb S^{n-1}$. Then
\[
        t\longmapsto |L_t^{s(L_t)}|^{-1}
\]
is convex on $[-a,a]$. If, in addition, $t\mapsto |L_t|$ is affine and $t\mapsto \VP(L_t)$ is constant on $[-a,a]$, then there exist $w\in\R^n$ and $\beta\in\R$ such that, for every $t\in[-a,a]$, one has $L_t=A_t(L_0)$, where $A_t:\mathbb{R}^n \to \mathbb{R}^n$ is the affine map defined by
\[
         A_t(x)=x+t(w\cdot x+\beta)\theta.
\]
\end{theorem}

The following corollary is a variation of the above theorem that we shall need, which is stated and proved in \cite[Corollary 2]{FMZ2012}.

\begin{corollary}[\cite{FMZ2012}]\label{cor:one-sided}
Let $(L_t)_{t\in[-a,a]}$ be a non-degenerate shadow system in $\R^n$ along a direction $\theta\in \mathbb S^{n-1}$, and suppose that $t\mapsto |L_t|$ is affine. If
\begin{equation}\label{l0min}
        \VP(L_0)=\min_{t\in[-a,a]}\VP(L_t),
\end{equation}
then $\VP(L_t)$ is constant on $[0,a]$ or constant on $[-a,0]$.
\end{corollary}

For completeness, we include a short verification of Corollary~\ref{cor:one-sided} from Theorem~\ref{thm:shadow}. In fact, the same argument gives the following slightly stronger conclusion: under the assumptions of Corollary~\ref{cor:one-sided}, the volume product is constant throughout the whole interval \([-a,a]\). We record this as the following lemma.

\begin{lemma}\label{simplified}
Under the same assumptions as in Corollary~\ref{cor:one-sided}, the quantity \(\VP(L_t)\) is constant on \([-a,a]\).
\end{lemma}

\begin{proof}[A verification of Corollary \ref{cor:one-sided} via Theorem \ref{thm:shadow}] Denote $f(t) = |L_t|$, $g(t) = |L_t^{s(L_t)}|$, and $h(t) = 1/g(t)$, for $t\in [-a,a]$. 
By Theorem \ref{thm:shadow}, $h$ is convex. By the assumption, $f,g,h>0$, the function $f$ is affine, and 
\[ \frac{f(t)}{h(t)} \ge \frac{f(0)}{h(0)}, \qquad \forall t\in[-a,a].\]
Then, 
   \[D(t) = h(0)f(t) - f(0)h(t) \ge 0, \qquad \forall t\in[-a,a].\]
Note that $D(t)$ is concave, since it is the difference between an affine function and a convex function. A nonnegative concave function on $[-a,a]$ vanishes at $t=0$ if and only if it vanishes identically. Therefore $D(t)\equiv 0$, which implies that $\VP(L_t)$ is a constant on $[-a,a]$.
\end{proof}

\section{Admissible speeds and shadow flows}\label{sec:admissible}

Let \(P\subset\mathbb R^3\) be a fixed three-dimensional polytope  with $V$ distinct vertices  \(x_1,\ldots,x_V\).  
We first introduce our key notion, that of an
admissible speed. An
$\alpha=(\alpha_1,\ldots,\alpha_V)\in\mathbb R^V$
will be called a {\it speed vector}. Given a direction
\(\theta\in \sn\) and a speed vector \(\alpha\in\mathbb R^V\), define a shadow system 
\[
        x_i(t)=x_i+t\alpha_i\theta,
        \qquad
        P_t=\conv\{x_i(t):1\leq i\leq V\}, \quad t\in [-c,c],
\]
where $c>0$. 
In the following, for a facet $F$ of $P$, $\lin(F-F)$ denotes the two-dimensional subspace of $\mathbb{R}^3$ parallel to $F.$ (We may use the convention $\conv(\varnothing)=\varnothing$ for completeness.)

We now introduce our first key definition of admissible speeds, which is designed to keep every old facet planar during the deformation. 

\begin{definition}[Admissible speed]\label{def:admissible}
A speed vector $\alpha \in \R^V$ is called \emph{  $\theta$-admissible} if the following condition holds for every facet $F$ of $P$. 
\begin{itemize}[leftmargin=2.2em]
\item[(1)] If $\theta\notin \lin(F-F)$, then there exists an affine function $\ell_F:\aff(F)\to\R$ such that $\ell_F(x_i)=\alpha_i$ for every vertex $x_i$ of $F$.
\item[(2)] If $\theta\in \lin(F-F)$, no constraint is imposed on the values $\alpha_i$ on the vertices of $F$.
\end{itemize}
When \(\theta\) is fixed or clear from the context, we simply say that \(\alpha\) is admissible.

Denote by $A_\theta(P) \subset \R^V$ the linear space of all $\theta$-admissible speeds. 
\end{definition}

It is simple to see that \(A_\theta(P)\) contains the following globally affine speeds as a subspace.

\begin{definition}[Trivial speeds, or globally affine speeds]\label{def:trivial-speeds}
A speed vector $\alpha\in\R^V$ is called \emph{trivial}, or \emph{globally affine}, if there exist $w\in\R^3$ and $\beta\in\R$ such that for $1\leq i\leq V$,
\[
        \alpha_i=w\cdot x_i+\beta.
\]
The set of all trivial speeds is denoted by
\[ 
\mathcal T(P) = \left\{
        (w\cdot x_i+\beta)_{1\leq i \leq V}:
        w\in\mathbb R^3,\ \beta\in\mathbb R
        \right\}.
\]
Then $\mathcal T(P)$ is a linear subspace of $A_\theta(P)$ for every direction $\theta$, and $\dim \mathcal T(P)=4$, since the space of affine functions on $\R^3$ has dimension 4. 
\end{definition}

\begin{remark}[Non-trivial speeds]\label{rem:nontrivial}
A \emph{non-trivial} admissible speed means that it is not globally affine. Note that $\mathcal T(P)$ depends on $P$ but is independent of $\theta$.
\end{remark}

For simplicity, we introduce the following notion of a shadow flow.

\begin{definition}[Shadow flow]\label{def:shadow-flow}
A \emph{shadow flow} is a shadow system along a direction $\theta$ with a $\theta$-admissible speed. More precisely, given a direction \(\theta\in \sn\) and a  $\theta$-admissible speed \(\alpha\in\mathbb R^V\), the shadow flow $(P_t)_{t\in[-c,c]}$ is defined by  
\[
   x_i(t) = x_i+t\alpha_i\theta, \qquad  P_t=\conv\{x_i(t)   :1\leq i\leq V\}, \quad t\in [-c,c],
\]
where $c>0$. 
\end{definition}

The term shadow flow may also be understood in a broader sense; see Remark \ref{rem:flow}. 

The main goal of this section is to establish Lemmas
\ref{lem:facet-persistence}, \ref{lem:volume-affine} and Proposition \ref {prop:face-affine-shadow}, for shadow flows. Since the first lemma is rather intuitive, we introduce the following definitions to make the argument clear. 
Throughout, we denote the vertices, edges, and facets  of the three-dimensional polytope \(P\) by 
\[
\mathcal V=\{x_1,\ldots,x_V\},\qquad
\mathcal E=\{E_1,\ldots,E_{n_1}\},\qquad
\mathcal F=\{F_1,\ldots,F_{n_2}\}.
\]
Let
\[
\mathcal I_0=\{1,\ldots,V\},\qquad
\mathcal I_1=\{1,\ldots,n_1\},\qquad
\mathcal I_2=\{1,\ldots,n_2\}
\]
be the corresponding sets of labels.

\begin{definition}[Face lattice] \label{def-fll}
    The \emph{face lattice} (or labeled  face lattice)   of \(P\), denoted by \(\mathcal S(P)\), is encoded by the index sets $\mathcal I_0, \mathcal I_1, \mathcal I_2$ together with the set-valued incidence maps
\[\mathcal I_0
\ \xrightarrow{\ \Phi_{1}\ }\
2^{\mathcal I_1},
\qquad
\mathcal I_1
\ \xrightarrow{\ \Phi_{2}\ }\
2^{\mathcal I_2},
\]
where
\[
\Phi_{1}(i)
=
\{j\in\mathcal I_1: x_i\in E_j\},\qquad \Phi_{2}(j)
=
\{k\in\mathcal I_2: E_j\subset F_k\}.
\]
\end{definition}

The above notation $\Phi_1$ and $\Phi_2$ induce the following set-valued incidence map
\[
\Phi_{0}(k)
=
\left\{
i\in \mathcal I_0:
\text{ there exists } j\in\mathcal I_1
\text{ such that }
j\in\Phi_1(i)
\text{ and }
k\in\Phi_2(j)
\right\}.
\]
Then, the vertex-facet incidence reads
\[
x_i\in F_k
\quad\Longleftrightarrow\quad
i\in \Phi_{0}(k).
\]

\begin{definition}[Persistence of the face lattice]
    We say that a family \((P_t)_{t\in [-c,c]}\) of polytopes with $P_0=P$
\emph{preserves the  face lattice} if 
\[
\mathcal S(P_t)=\mathcal S(P) \qquad {\rm for~all~} t.
\]
More precisely, $\mathcal S(P_t)= \mathcal S(P)$ means that the vertices, edges, and facets of \(P_t\) can be labeled as
\[
\{x_i(t):i\in\mathcal I_0\},\qquad
\{E_j(t):j\in\mathcal I_1\},\qquad
\{F_k(t):k\in\mathcal I_2\},
\]
where \(x_i(0)=x_i\), \(E_j(0)=E_j\), and \(F_k(0)=F_k\), such that
\[
x_i(t)\in E_j(t)
\quad\Longleftrightarrow\quad
x_i\in E_j \quad\Longleftrightarrow\quad j\in \Phi_1(i),
\]
and
\[
E_j(t)\subset F_k(t)
\quad\Longleftrightarrow\quad
E_j\subset F_k \quad\Longleftrightarrow\quad k\in \Phi_2(j),
\]
for all \(i\in\mathcal I_0\), \(j\in\mathcal I_1\), and
\(k\in\mathcal I_2\).  
\end{definition}
\begin{remark}
    Our labeled face lattice is just the usual face lattice, i.e. the inclusion poset of all faces, with fixed labels attached to the faces.
Usually, saying that two polytopes have the same face lattice means that there is an inclusion-preserving isomorphism between their face lattices. 
Here the labels play the role of this isomorphism: faces with the same label are identified, and the incidence relations among the labeled faces are required to be the same.
\end{remark}  

Next, we introduce a lemma showing that the shadow flow $(P_t)_{t\in[-c,c]}$ preserves the face lattice for short time.
\begin{lemma}[Persistence of the face lattice]\label{lem:facet-persistence}
Let \(P\subset\mathbb R^3\) be a convex polytope with vertices \(x_1,\ldots,x_V\). Let $\alpha \in A_\theta(P)$, and $(P_t)$ be a shadow flow given by
\[
        x_i(t)=x_i+t\alpha_i\theta,
        \qquad
        P_t=\operatorname{conv}\{x_i(t):1\le i\le V\}.
\]
Then, there exists \(c>0\) such that $(P_t)_{t\in[-c,c]}$ preserves the face lattice. 
\end{lemma}

\begin{proof} 
Let \(F_k \in \mathcal F \) be any facet of \(P\),  $k\in \I_2$. We shall make use of the notation $\Phi_1$ and $\Phi_2$ in the definition of face lattice, and the induced vertex-facet incidence map $\Phi_0$. Recall that $i\in \Phi_0(k)$ means that $x_i$ is a vertex of $F_k$. 

Since $\alpha$ is \(\theta\)-admissible, the moved points 
\[
        x_i(t), \qquad i \in \Phi_{0}(k),
\]
remain contained in a single plane, which is denoted by \(\Pi_{F_k}(t)\). 
In fact, if Case (1) in Definition \ref{def:admissible} holds, after extending $\ell_{F_k}$ to an affine function on $\mathbb{R}^3$, there are $w_k\in \R^3$ and $\beta_k\in \R$, such that
\[
        x_i(t) =x_i + t\alpha_i\theta = x_i+ t(w_k\cdot x_i+\beta_k)\theta ,\qquad i\in \Phi_0(k).
\]

Let \(v_k\) be the unit outer normal vector of \(P\) corresponding to the facet \(F_k\). Set
\[
v_k(t)
=
(1+t\,\theta\cdot w_k)v_k
-
t(\theta\cdot v_k)w_k.
\]
Then a direct computation gives
\begin{equation}\label{eq-keephpk}
x_i(t)\cdot v_k(t)
=
h_P(v_k)+t ((\theta\cdot w_k )h_P(v_k)+(\theta\cdot v_k )\beta_k),
\qquad \forall i\in \Phi_0(k), 
    \end{equation}
which is independent of $i$. 

If Case (2) in Definition \ref{def:admissible} holds, then \(\Pi_{F_k}(t) =  \aff F_k\) for all $t$, and $x_i(t)=x_i + t\alpha_i\theta$ 
keeps in the same plane $\aff F_k$. Set $v_k(t)=v_k$ and $x_i(t)\cdot v_k(t)=h_P(v_k).$

First, for sufficiently small $|t|$, the points $x_i(t)$ remain extreme points (vertices) of $P_t$. Second, since
$
x_l\cdot v_k<h_P(v_k)$ for all $l\notin \Phi_0(k),$
by the continuity of \(x_i(t)\), \(x_l(t)\) and \(v_k(t)\), and by \eqref{eq-keephpk}, we have, for sufficiently small \(|t|\),
\[
x_l(t)\cdot v_k(t)<x_i(t)\cdot v_k(t),
\qquad i\in \Phi_0(k),\quad l\notin \Phi_0(k).
\]
Hence, for every pair \((i,k)\) with \(i\in \Phi_0(k)\), we have
\begin{equation}\label{eq-hptvkt}
x_i(t)\cdot v_k(t)=h_k(t),    
\end{equation} 
where $h_k(t) = h_{P_t}(v_k(t))$ equals the right-hand side of \eqref{eq-keephpk}.  
Thus
\[  
F_k(t)=\conv\{x_i(t):i\in \Phi_0(k)\}
\]
is still a facet of \(P_t\).
Third, we consider the one-dimensional faces (edges) of $P_t$. Suppose $E_j = F_{k_1}\cap F_{k_2} = \conv\{x_{i_1},x_{i_2}\}$ with $\{k_1,k_2\} = \Phi_2(j)$ and $j\in \Phi_1(i_1)\cap \Phi_1(i_2).$ Then, by the conclusions above, 
\[ x_{i_1}(t) \cdot v_{k_s}(t) = x_{i_2}(t) \cdot v_{k_s}(t)= h_{P_t}(v_{k_s}(t)), \qquad s=1,2,  \]
which means that $E_j(t) = F_{k_1}(t)\cap F_{k_2}(t) = \conv\{x_{i_1}(t),x_{i_2}(t)\}$ is an edge of $P_t$. 

It remains to show that no additional faces occur. 
Define
\[
Q_t=
\bigcap_{k\in\mathcal I_2}
\{x\in\mathbb R^3:x\cdot v_k(t)\le h_k(t)\}.
\]
By \eqref{eq-hptvkt} and the inequality in the preceding line, 
\[
P_t\subset Q_t,
\]
and
\[
x_i(t)\cdot v_k(t)=h_k(t)
\quad\Longleftrightarrow\quad
i\in\Phi_0(k).
\]

We claim that \(Q_t=P_t\) for all sufficiently small \(|t|\). Since \(Q_0=P\) is bounded, the outer facet normals of \(P\) positively span \(\mathbb R^3\). By continuity, the same is true for the vectors \(v_k(t)\) for small \(|t|\). Since the numbers \(h_k(t)\) are also
uniformly bounded, the polytopes \(Q_t\) are uniformly bounded.

Suppose, to the contrary, that \(Q_t\ne P_t\) for arbitrarily small \(t\). Then there are \(t_n\to0\) and a vertex \(y_n\) of \(Q_{t_n}\) such that
\[
y_n\notin P_{t_n}.
\]
By uniform boundedness, after passing to a subsequence,
\[
y_n\to y\in Q_0=P.
\]
At the vertex \(y_n\) of $Q_{t_n}$, choose a triple
\(k_1,k_2,k_3\in\mathcal I_2\) such that 
\[
v_{k_1}(t_n),\ v_{k_2}(t_n),\ v_{k_3}(t_n)
\]
are the normals of facets containing \(y_n\), and are linearly independent. Since there are only finitely many triples, we may pass to a subsequence and assume that the same triple \(k_1,k_2,k_3\) works for every \(n\). Then
\[
y_n\cdot v_{k_r}(t_n)=h_{k_r}(t_n),\qquad r=1,2,3.
\]
Letting \(n\to\infty\), we get
\[
y\in F_{k_1}\cap F_{k_2}\cap F_{k_3}.
\]
This intersection is a non-empty face of \(P\). Since \(k_1,k_2,k_3\) are distinct and every edge of a three-dimensional polytope is incident to exactly two facets, this face must be a vertex, say
\[
F_{k_1}\cap F_{k_2}\cap F_{k_3}=\{x_i\}.
\]
Thus \(i\in\Phi_0(k_r)\) for \(r=1,2,3\), and therefore
\[
x_i(t_n)\cdot v_{k_r}(t_n)=h_{k_r}(t_n),
\qquad r=1,2,3.
\]
Both \(y_n\) and \(x_i(t_n)\) lie in the three affine planes indexed by \(k_1,k_2,k_3\). Since their normals are linearly independent, these three planes have a
unique intersection point. Hence
\[
y_n=x_i(t_n)\in P_{t_n},
\]
contradicting the choice of \(y_n\).

Therefore \(Q_t=P_t\) for all sufficiently small \(|t|\). Thus, every facet of \(P_t\) is one of the facets \(F_k(t)\). Since every zero- or one-dimensional face of $P_t$ is a face of some facet, no new faces occur. By the above discussion, the labeled face lattice is preserved.
\end{proof}

{A byproduct of the above lemma is that the faces 
\[
E_j(t) = \conv \{x_i(t) : i\in \mathcal I_0, j\in \Phi_1(i) \}, \quad
F_k(t) = \conv \{x_i(t) : i \in \Phi_{0}(k)\}, \quad    j\in\mathcal I_1,\quad k\in\mathcal I_2,
\]
depend continuously on $t$.}

The following lemma shows that the volume of $P_t$ is affine in $t$ once the face lattice is preserved.

\begin{lemma} \label{lem:volume-affine}
If there exists $c>0$ such that shadow flow
\[
        P_t=\conv\{x_i+t\alpha_i\theta: 1\leq i\leq V\}
\]
preserves the face lattice, for all \(t \in [-c,c]\), then
\[
        t\longmapsto |P_t|
\]
is affine on \( [-c,c]\).
\end{lemma}

\begin{proof}
For each facet \(F\) of \(P\), choose a triangulation using only vertices of \(P\). Let \(\Gamma\) denote
the collection of all vertex-index triples
\[
(i,j,k)
\]
which occur as triangles in these triangulations of the facets of \(P\).
Since the labeled face lattice is preserved, the same collection \(\Gamma\) gives triangulations of the corresponding facets of \(P_t\).

Denote
\[
        p(t)=\frac1V\sum_{i=1}^V x_i(t).
\]
This point belongs to \(\operatorname{int}P_t\).  Indeed, every facet of \(P_t\) misses at least one vertex, 
and \(p(t)\) is a convex combination of all vertices with positive coefficients. Hence
\[p(t)\cdot v_l(t) < h_{P_t}(v_l(t)),\]
for every outer normal $v_l(t)$ of $P_t$. 

Consider the tetrahedra obtained by joining \(p(t)\) to the triangles in the triangulation \(\Gamma\):
\[
\conv\, \{p(t),x_i(t),x_j(t),x_k(t)\},
\qquad (i,j,k)\in\Gamma.
\]
These tetrahedra cover \(P_t\) and have disjoint interiors. Therefore
\[
        |P_t|
        =
        \frac16
        \sum_{(i,j,k)\in\Gamma}
        \left|
        \det\bigl(
             x_i(t)-p(t),
             x_j(t)-p(t),
             x_k(t)-p(t)
        \bigr)
        \right|.
\]

Put
\[
    \bar\alpha=\frac1V\sum_{r=1}^V\alpha_r .
\]
Then
\[
    p(t)=p(0)+t\bar\alpha\theta,
\]
and hence
\[
x_i(t)-p(t)
=
x_i-p(0)+t(\alpha_i-\bar\alpha)\theta .
\]
Thus each column in the determinant has a constant part plus a multiple of the fixed direction \(\theta\).

For each index triple \((i,j,k) \in \Gamma\), define
\[
        D_{ijk}(t)
        =
        \det\bigl(
             x_i(t)-p(t),
             x_j(t)-p(t),
             x_k(t)-p(t)
        \bigr).
\]
By multilinearity of the determinant, \(D_{ijk}(t)\) is a polynomial in \(t\).  Every term of degree at least two contains at least two columns proportional to \(\theta\), and therefore vanishes. Hence \(D_{ijk}(t)\) is affine in \(t\).

Moreover, \(D_{ijk}(t)\neq0\) for every \(t\in I:=[-c, c]\), because \(p(t)\) is an interior point of \(P_t\) and the triangle \(\conv\{x_i(t),x_j(t),x_k(t)\}\) lies in a boundary facet.
Since \(I\) is connected, the sign of \(D_{ijk}(t)\) is constant on \(I\). Therefore
\[
        |D_{ijk}(t)|
\]
is also affine in \(t\). Summing over the fixed finite set \(\Gamma\), we conclude that \(t\mapsto |P_t|\) is affine on \(I\).
\end{proof}

The following proposition shows that the shadow flows are in fact volume-affine. 

\begin{proposition}\label{prop:face-affine-shadow}
Let $P\subset\R^3$ be a convex polytope, and let  $\alpha \in A_\theta(P)$. Then there is $c>0$ such that the shadow flow $(P_t)_{t\in[-c,c]}$ is  non-degenerate and preserves the face lattice, and such that $t\mapsto |P_t|$ is affine on $[-c,c]$. 
\end{proposition}

\begin{proof}
By Lemma \ref{lem:facet-persistence}, the shadow flow $(P_t)_{t\in[-c,c]}$  preserves the face lattice; in particular the bodies remain full-dimensional, so the shadow flow is non-degenerate. Lemma \ref{lem:volume-affine} gives the affine dependence of the volume. 
\end{proof}

\begin{remark} \label{rem:flow}
One may imagine successively applying short-time shadow flows to the polytope itself or to its Santal\'o polar, restarting the process whenever the face lattice changes. This should lead to a terminal body, namely a polytope for which both the body and its Santal\'o polar admit only trivial admissible speeds. In the present paper, however, we do not carry out this forward construction explicitly; instead, we prove directly that any minimizer of the Mahler volume (see Section \ref{Sec:tools} for definition) must already be a terminal body.
This is shown in Section \ref{sec:lem:local-exclusion}.

The  conceptual point of the shadow flow is that we do not seek an explicit geometric deformation that sends a polytope directly to a prescribed singular body, such as a simplex, a cube, an octahedron, or generally a Hanner polytope. 
Rather, at each stage it is governed by the admissible-speed structure of the current polytope. 
This makes it possible to start from an arbitrary polytope and continue the process until a terminal structure is reached, thereby overcoming the traditional restriction of shadow-system arguments to special cases.
In three dimensions, the possible terminal structures are then identified by counting arguments. In the general form, Section \ref{sec:combinatorial} shows that terminal polytopes must be simplices. In the origin-symmetric form, Section \ref{sec:final} shows that terminal polytopes must be affine images of cubes or octahedra. 

Most of the results established for shadow flows, except the determination of the terminal bodies, also work in $\R^n$.
\end{remark}

\section{Minimizers admit only trivial shadow flows}\label{sec:lem:local-exclusion}

Recall that we call \((P_t)_{t\in[-c,c]}\) a {shadow flow}, if it is a shadow system along the direction
\(\theta\in \sn\) with the speed vector
\(\alpha\in\mathbb R^V\) $\theta$-admissible. A shadow flow, with its underlying direction \(\theta\) understood, is called non-trivial or trivial, if its speed vector is non-trivial or trivial, respectively.

The main purpose of this section is to establish
Lemma \ref{lem:minimizer-exclusion}; 
for this, we first need the following local version.

\begin{lemma} \label{lem:local-exclusion}
Let  
\[        
P_t=\conv\{x_i+t\alpha_i\theta:1\leq i\leq V\},\qquad t\in[-c,c], 
\]
be a shadow flow, where $c>0$ is sufficiently small so that Proposition \ref{prop:face-affine-shadow} applies. If
\[
        \VP(P_0)=\min_{t\in[-c,c]}\VP(P_t),
\]
then the speed vector $\alpha$ is trivial. 
\end{lemma}

\begin{proof}
By Proposition~\ref{prop:face-affine-shadow}, the family
\((P_t)_{t\in[-c,c]}\) is a non-degenerate volume-affine shadow flow, and \(P_t\) preserves the face lattice for all \(t\in[-c,c]\). Since
\[
\VP(P_0)=\min_{t\in[-c,c]}\VP(P_t),
\]
Lemma~\ref{simplified} implies that \(\VP(P_t)\) is constant on \([-c,c]\).
Then, by Theorem~\ref{thm:shadow}, there exists an affine map
\(A_t:\mathbb R^3\to\mathbb R^3\) of the form
\[
A_t(x)=x+t(w\cdot x+\beta)\theta
\]
such that
\[
        P_t=A_t(P_0).
\]
By the continuity in $t$ of both $A_t$ and the vertices $x_i(t) = x_i+t\alpha_i \theta$, we must have 
\[ A_t(x_i) = x_i(t),\]
when $|t|>0$ is sufficiently small.
Hence the speed vector \((w\cdot x_i+\beta)_{1\leq i\leq V}\) is globally affine.

\end{proof}

By Lemma \ref{lem:bounded-class}, $\mathcal P(\cdot)$ attains its infimum in $\mathcal C_N$.
The following lemma shows that a minimizer of $\mathcal P$ on $\mathcal C_N$ can admit only a trivial shadow flow. 

\begin{lemma} \label{lem:minimizer-exclusion}
Let $\mathcal C_N$ denote the class of full-dimensional
convex polytopes in $\mathbb{R}^3$ with at most $N$ vertices. Fix $N\geq4$, and let $Q\in\mathcal C_N$ be a minimizer of $\VP$ over $\mathcal C_N$. Then neither $Q$ nor its Santal\'o polar $Q^{s(Q)}$ admits a non-trivial shadow flow along any direction $\theta\in \sn$.
\end{lemma}

\begin{proof}
First suppose that $Q$ admits such a shadow flow $Q_t$ along some direction $\theta\in \mathbb{S}^2$. By Proposition \ref{prop:face-affine-shadow}, after shrinking the interval the face lattice is preserved. Hence $Q_t$ has the same number of vertices as $Q$, and $Q_t\in\mathcal C_N$ for small $|t|$. By minimality,
\[
    \VP(Q_t)\geq \VP(Q)=\VP(Q_0).
\]
Thus $Q_0$ is an interior minimum along the shadow flow. Lemma \ref{lem:local-exclusion} forces the speed to be globally affine, contradicting non-triviality.

Now put $L=Q^{s(Q)}$, and suppose that $L$ admits a non-trivial shadow flow $L_t$. Again shrink the interval so that Proposition \ref{prop:face-affine-shadow} applies. We claim that $L_0$ is an interior minimum along the shadow flow $L_t.$ 

By Lemma \ref{lem:polarity-face-lattice}, polarity with respect to  an interior point, in particular with respect to the Santal\'o point, reverses the face lattice. 
Since the face lattice of $L_t$ is preserved for sufficiently small $t$, we have
\[
        V(L_t^{s(L_t)})=F(L_t)=F(L_0)=V(Q)\leq N,
\]
and hence
\begin{equation}\label{eq:Mt-in-CN}
        L_t^{s(L_t)}\in\mathcal C_N\qquad\text{for all small }t.
\end{equation}
Since $Q\in \mathcal C_N$ minimizes $\mathcal P$ over $\mathcal C_N$,
\begin{equation}\label{eq:minimality-Q-Mt}
        \VP(Q)\leq \VP(L_t^{s(L_t)}).
\end{equation}

By Lemma \ref{lem:polar-nonincreasing}, applied to $L_t$, we have
\begin{equation}\label{eq:Mt-less-Lt}
        \VP(L_t^{s(L_t)})\leq \VP(L_t).
\end{equation}
Recalling that $L_0 = Q^{s(Q)}$, by \eqref{eq:minimality-Q-Mt} and \eqref{eq:Mt-less-Lt}, and Lemma \ref{lem:polar-nonincreasing} applied to $Q$, we obtain
\begin{equation}\label{eq:chain-L}
        \VP(L_0)\leq \VP(Q)\leq \VP(L_t).
\end{equation}
 Thus $L_0$ is an interior minimum along the shadow flow $L_t$. Now Lemma \ref{lem:local-exclusion} forces the speed of $L_t$ to be trivial, again contradicting the assumption that $L_t$ is a non-trivial shadow flow.
\end{proof}

\section{A counting argument in  three dimensions}\label{sec:combinatorial}

Let $P$ be a three-dimensional polytope. In this section, let 
\[
        V=V(P),\qquad E=E(P),\qquad F=F(P)
\]
denote the numbers of vertices, edges, and facets. 
Let
\[
        \Delta(P)=\max\{\#\Phi_0(k):\, k\in \mathcal I_2\}
\]
denote the maximum number of vertices of a facet of $P.$
Let
\[
        d(P)=\max\{\#\Phi_1(i):\, i\in \mathcal I_0\}
\]
denote the maximum number of edges incident to a vertex of $P.$
In three dimensions, $\#\Phi_1(i)$ is also the number of facets incident to $x_i$, because the vertex figure at $x_i$ is a polygon whose edges correspond to incident facets.

The well-known Euler's formula reads 
\[ V+F-E=2.\]
We shall derive the following counting lemma, which relates the dimension of the linear space of admissible speeds to the quantities 
\(V\), \(F\), and \(\Delta(P)\).

\begin{lemma}\label{lem:dimension-count}
Let \(P\) be a three-dimensional convex polytope. Choose a facet \(G_0\) with \(\Delta(P)\) vertices, and choose a direction
\[
    \theta \in \lin(G_0-G_0)\cap \sn .
\]
Let \(A_\theta(P)\) be the vector space of \(\theta\)-admissible speeds. Then
\begin{equation}\label{adimest}
        \dim A_\theta(P) \ge F(P)-V(P)+\Delta(P)+1.       
\end{equation}
Consequently, if $\dim A_\theta(P) = 4$, then
\begin{equation}\label{detaplower}
        \Delta(P)\le V(P)-F(P)+3.      
\end{equation}
\end{lemma}

\begin{proof}
Write
\[
        V=V(P),\qquad E=E(P),\qquad F=F(P).
\]
Let the vertices of \(P\) be \(x_1,\ldots,x_V\). For a facet \(G\), let
\[
        I(G)=\{i:x_i\in G\},
        \qquad
        m(G)=\#I(G).
\]
Thus \(m(G)\) is the number of vertices of \(G\).

We first describe explicitly the linear conditions imposed by one facet. Let
\[
       \A (\aff G)
\]
denote the vector space of real-valued affine functions on the affine plane \(\aff G\). Since \(\aff G\) is two-dimensional,
\[
        \dim \A(\aff G)=3.
\]
Consider the evaluation map
\[
        \operatorname{ev}_G:\A(\aff G)\longrightarrow \mathbb R^{m(G)},
        \qquad
        \ell\longmapsto \bigl(\ell(x_i)\bigr)_{i\in I(G)}.
\]
This map is injective. Indeed, if an affine function on \(\aff G\) vanishes at all vertices of \(G\), then in particular it vanishes at three non-collinear vertices of \(G\). An affine function on a plane which vanishes at three non-collinear points is identically zero. Therefore
\[
        \dim \operatorname{im}(\operatorname{ev}_G)=3.
\]
Hence the subspace
\[
        \operatorname{im}(\operatorname{ev}_G)\subset \mathbb R^{m(G)}
\]
has codimension
\[
        m(G)-3.
\]

Now let \(\alpha=(\alpha_1,\ldots,\alpha_V)\in\mathbb R^V\) be an admissible speed vector. Its restriction to the vertices of \(G\) is
\[
        \alpha|_G=(\alpha_i)_{i\in I(G)}\in \mathbb R^{m(G)}.
\]
If
\[
        \theta\notin \lin(G-G),
\]
then the definition of \(\theta\)-admissibility says exactly that there exists an affine function \(\ell_G\in \A(\aff G)\) such that
\[
        \ell_G(x_i)=\alpha_i,
        \qquad i\in I(G).
\]
Equivalently,
\[
        \alpha|_G\in \operatorname{im}(\operatorname{ev}_G).
\]

If instead
\[
        \theta\in \lin(G-G),
\]
then \(G\) imposes no constraint by the definition of admissibility.

Let
\[
        \mathcal F_\theta
        =
        \{G\in \mathcal F:\theta\notin \lin(G-G)\}
\]
be the set of facets not parallel to \(\theta\). Define the linear map
\[
        T:\mathbb R^V
        \longrightarrow
        \bigoplus_{G\in\mathcal F_\theta}
        \mathbb R^{m(G)}/\operatorname{im}(\operatorname{ev}_G)
\]
by
\[
        T(\alpha)
        =
        \bigl(\alpha|_G+\operatorname{im}(\operatorname{ev}_G)\bigr)_{G\in\mathcal F_\theta}.
\]
By construction,
\[
        A_\theta(P)=\ker T.
\]
Therefore, by rank-nullity,
\[
        \dim A_\theta(P)
        =
        V-\operatorname{rank} T.
\]
Since the rank of \(T\) is at most the dimension of its target space,
\[
        \operatorname{rank} T
        \le
        \sum_{G\in\mathcal F_\theta}
        \bigl(m(G)-3\bigr).
\]
Consequently,
\begin{equation}\label{dimest}
        \dim A_\theta(P)
        \ge
        V-\sum_{G\in\mathcal F_\theta}
        \bigl(m(G)-3\bigr).             
\end{equation}

It remains to estimate the sum in \eqref{dimest}. Since
\[
        \theta\in \lin(G_0-G_0),
\]
the chosen facet \(G_0\) is parallel to \(\theta\), and hence \(G_0\notin \mathcal F_\theta\). Since \(m(G_0)=\Delta(P)\), we get
\begin{equation}\label{mgupbound}
        \sum_{G\in\mathcal F_\theta}\bigl(m(G)-3\bigr)
        \le
        \sum_{G\text{ facet of }P}\bigl(m(G)-3\bigr)
        -\bigl(\Delta(P)-3\bigr).                   
\end{equation}

Now compute the full facet sum. Each facet \(G\) is a polygon with \(m(G)\) vertices and \(m(G)\) edges. Therefore
\[
        \sum_{G\text{ facet of }P} m(G)=2E,
\]
because every edge of a three-dimensional polytope belongs to exactly two
facets. Hence
\begin{equation}\label{eulerest}
        \sum_{G\text{ facet of }P}\bigl(m(G)-3\bigr)
        =
        2E-3F.                                     
\end{equation}

Substituting the Euler's formula into \eqref{eulerest} gives
\begin{equation}\label{2e3fest}
        2E-3F
        =
        2(V+F-2)-3F
        =
        2V-F-4.                                     
\end{equation}
Combining \eqref{mgupbound} and \eqref{2e3fest}, we obtain
\begin{equation}\label{sumg}
        \sum_{G\in\mathcal F_\theta}\bigl(m(G)-3\bigr)
        \le
        (2V-F-4)-(\Delta(P)-3)
        =
        2V-F-\Delta(P)-1.                          
\end{equation}
Putting \eqref{sumg} into \eqref{dimest}, we get
\[
        \dim A_\theta(P)
        \ge
        V-\bigl(2V-F-\Delta(P)-1\bigr)
        =
        F-V+\Delta(P)+1.
\]
This proves \eqref{adimest}.

 Finally, when $\dim A_\theta(P)=4$, \eqref{detaplower} is immediate.
\end{proof}

\begin{remark} The inequality \eqref{dimest} can also be understood directly: on each facet  $G\in \F_\theta$, there are $(m(G)-3)$ constraints, since the speeds at any triple of  vertices determine the affine function $\ell_G$ on the plane $\aff G$.

\end{remark}

The following lemma implies that a three-dimensional convex polytope with only trivial shadow flows must be a tetrahedron.

\begin{lemma}\label{lem:combinatorial-alternative}
Let $P$ be a three-dimensional convex polytope. 
If both  $P$ and its Santal\'o polar $P^{s(P)}$ admit only trivial shadow flows in every direction $\theta$,  then  $P$ must be a  tetrahedron. 
\end{lemma}

\begin{proof} Choose a facet \(G_0\) with \(\Delta(P)\) vertices, and choose a  direction
\[
        \theta \in \lin(G_0-G_0)\cap \sn.
\]
Since $P$ and $P^{s(P)}$ admit only trivial shadow flows in the direction $\theta$, and recall from Definition \ref{def:trivial-speeds} that $\dim \mathcal T(P) = 4$, we have
\[ \dim A_\theta(P) =  \dim A_\theta(P^{s(P)}) = 4.\]

By Lemma \ref{lem:dimension-count}, we have
\begin{equation}\label{eq:A}
        \Delta(P) \le V(P)-F(P)+3.
\end{equation}
By Lemma \ref{lem:polarity-face-lattice}, the face lattice of $P^{s(P)}$ is dual to that of $P$, so
\[
        V(P^{s(P)})=F(P),\qquad F(P^{s(P)})=V(P).
\]
A facet of $P^{s(P)}$ corresponding to a vertex $v$ of $P$ has as many vertices as the number of facets of $P$ incident to $v$. Therefore
\begin{equation}\label{eq:polar-delta-degree}
        \Delta(P^{s(P)})=d(P).
\end{equation}
Applying Lemma \ref{lem:dimension-count} to polytope $P^{s(P)}$, we get 
\begin{equation}\label{eq:B}
        d(P) \le F(P)-V(P)+3.
\end{equation} 

Note that every facet has at least three vertices and every vertex has degree at least three. Combining these with \eqref{eq:A} and \eqref{eq:B}, we infer that  
   \[  \Delta(P) = d(P) = 3 \qquad {\rm and}  \qquad   V=F.\] 
 Thus every facet is triangular and every vertex has degree three. Consequently,
\[
        2E=3F,
        \qquad
        2E=3V.
\]
Now Euler's formula yields
\[
        2=V-E+F=2V-\frac32V=\frac12V,
\]
and hence  $V=4$. Therefore $P$ is a tetrahedron.
\end{proof}

\section{Proof of the Main Result}\label{sec:main-theorem}
\subsection{The Inequality Part}\label{sec:main6.1}
\begin{theorem}[The three-dimensional Mahler conjecture for  general convex bodies]\label{thm:main}
Every convex body $K\subset\R^3$ satisfies
\[
        \VP(K)\geq \frac{64}{9}.
\]
\end{theorem}

\begin{proof}
We first prove the assertion for polytopes. Let $N\geq4$, and let $\mathcal C_N$ be the class of three-dimensional convex polytopes with at most $N$ vertices. By Lemma \ref{lem:bounded-class}, there exists a minimizer $Q\in\mathcal C_N$ of $\VP$ over $\mathcal C_N$.

By Lemma \ref{lem:minimizer-exclusion}, $Q$ and $Q^{s(Q)}$ admit only trivial shadow flows in every direction $\theta$. By Lemma \ref{lem:combinatorial-alternative}, $P$ must be a tetrahedron.

Clearly, every tetrahedron (three-dimensional simplex) belongs to $\mathcal C_N$. Since $\VP$ is affine invariant and the simplex value is $64/9$, the minimum of $\VP$ on $\mathcal C_N$ is exactly $64/9$. Thus every polytope with at most $N$ vertices satisfies $\VP(P)\geq64/9$. Since $N$ is arbitrary, the lower bound holds for every three-dimensional convex polytope.

Now let $K\subset\R^3$ be an arbitrary convex body. Choose a sequence $(P_m)_{m\ge1}$ of three-dimensional polytopes such that $P_m\to K$ in the Hausdorff metric as $m\to \infty$. By the polytope case,
\[
        \VP(P_m)\geq\frac{64}{9}
\]
for every $m$. Hausdorff-continuity of the volume product, including the continuity of the Santal\'o point and of the corresponding polar volume, gives
\[
        \VP(K)=\lim_{m\to\infty}\VP(P_m)\geq\frac{64}{9}.\qedhere
\]
\end{proof}

\subsection{The Equality Part} 
\label{sec:main6.2}
The equality characterization will follow from a stability result, for which we need the following affine Banach--Mazur distance. Let $K, L \subset \R^3 $ be two convex bodies, their {\it affine Banach-Mazur distance} is given by 
\[
\begin{aligned}
d_{BM}(K,L)
= \inf\big\{c\geq 1 : A(L)\subset B(K)\subset cA(L), \,\,\text{for } A,  B \in \Aff(3) \big\},
\end{aligned}
\]
where $\Aff(3)$ denote the group of invertible affine isomorphisms of $\R^3$.

If both $K$ and $L$ are symmetric convex bodies, this is just the classical Banach--Mazur distance. It
is well known that the Banach–Mazur distance is a multiplicative metric, i.e., that log $d_{BM}$ is a metric.

We begin with a preliminary lemma, which follows from convexity.

 \begin{lemma}\label{sliding lemma}
Let $K\subset \R^3$ be a convex body, and let $z_0=c(K)$ be its centroid. For $0\le \tau\le 1,$ denote
$z_\tau=(1-\tau)\p z_0+\tau \p s(K).
$
Then, 
\[
\mathcal{P}(K^{z_\tau})\le \mathcal{P}(K^{z_0}),\qquad 0\le \tau\le1.
\]
\end{lemma}

\begin{proof}
It is well known that $|K^z|$ is convex with respect to $z.$ Since its minimum is attained at $s(K)$, we have
\[
|K^{z_\tau}|\le (1-\tau)|K^{z_0}|+\tau |K^{s(K)}|\le |K^{z_0}|.
\]
Furthermore, $(K^{z_\tau})^{z_\tau}=K$, and hence
\[
\cP(K^{z_\tau})\le |K^{z_\tau}|\,|(K^{z_\tau})^{z_\tau}|=|K^{z_\tau}|\,|K|\le |K^{z_0}|\,|K|.
\]
By \eqref{centr-santalo}, 
$z_0=s(K^{z_0}).$ Then $\cP(K^{z_0})=|K^{z_0}|\,|(K^{z_0})^{s(K^{z_0})}| = |K^{z_0}| |K|$, which implies $\mathcal{P}(K^{z_\tau})\le \cP(K^z)$. 
\end{proof}
 
For $N\ge 4$, let
\[
  \widetilde{\mathcal C}_N :=\mathcal C_N/\Aff(3).
\]
An element of $\widetilde{\mathcal C}_N$ is written as $[P]$, the affine equivalence class of $P\in\mathcal C_N$. We equip $\widetilde{\mathcal C}_N$ with the metric
\[
  \rho_{BM}([P],[Q]) :=\log d_{BM}(P,Q),
\]
where $d_{BM}$ is the affine Banach--Mazur distance. Under this metric, $\widetilde{\mathcal C}_N$ is compact. Indeed, by choosing John-position representatives, every class has a representative satisfying $B_2^3\subset P\subset 3B_2^3.$ Hausdorff convergence in this normalized class implies convergence in $\rho_{BM}$.

Since $\mathcal P$ is affine invariant, it corresponds to a continuous function on $\widetilde{\mathcal C}_N$, still denoted by $\mathcal P$. For $a\ge 64/9$, define
\[
  \widetilde{\mathcal C}_{N,a}=\{[P]\in \widetilde{\mathcal C}_N:\mathcal P(P)\le a\}.
\]
We aim to show that $\widetilde{\mathcal C}_{N,a}$ is connected for each $a\ge 64/9$. To this end, we first prove the following lemma, which describes how the Mahler volume decreases along a shadow flow.

\begin{lemma}\label{lem:strict-d-flow}
Assume $(P_t)_{t\in[-c,c]}$ is a non-trivial shadow flow, where $c>0$ is sufficiently small so that  $(P_t)_{t\in[-c,c]}$ preserves the face lattices of $P$. Then, 
\[
\text{either }\,\VP(P_{t})<\VP(P) \,\text{ for all }\, t\in [-c,0), \quad \text{   or   }\quad \VP(P_{t})<\VP(P) \,\text{ for all } t\in (0,c].
\]
\end{lemma}

\begin{proof}
Denote
\[
f(t)=|P_t|,\qquad   h(t)=\frac{1}{|P_t^{s(P_t)}|},
\qquad t\in[-a,a].
\]
By Theorem \ref{thm:shadow}, \(h\) is convex. By assumption, \(f\) is affine. Define
\[
D(t)=h(0)f(t)-f(0)h(t).
\]
Since \(h(t)>0\), we have
\[
\mathcal{P}(P_t)\le \cP(P_0)\quad 
\Longleftrightarrow \quad 
\frac{f(t)}{h(t)}\le \frac{f(0)}{h(0)}\quad 
\Longleftrightarrow \quad D(t)\le 0.
\]
Since \(D\) is concave and \(D(0)=0\), at least one of the two half-neighborhoods of \(0\) satisfies \(D\le 0\). Thus there exist a sign \(\sigma\in\{-1,1\}\)  such that
\[
D(\sigma t)\le 0,\qquad \forall\, t\in [0, c].
\] 
We may assume, after possibly replacing $t$ by $-t$ that, $D(t)\le 0,$ for all $t\in [0,c].$ Then, if
 $D(t_0)=0$ for some $t_0\in (0, c],$ by concavity and non-positivity of $D$, we must have $D=0$ on $[0, t_0].$ Then $\VP(P_t)= \VP(P_0)$ on $[0, t_0].$ 
The rigidity part of Theorem~\ref{thm:shadow}, applied to this interval, then forces the speed to be globally affine, which contradicts to the assumption that $P_t$ is a non-trivial shadow flow.
Therefore $D(t)<0$, equivalently $\cP(P_t) < \cP(P)$, for all $t\in [0, c].$
\end{proof}

\begin{proposition}\label{prop:strict-sublevel}
The set $\widetilde{\mathcal C}_{N,a}$ is connected for each $a\ge 64/9$.
\end{proposition}

\begin{proof}   
It suffices to show that, for every connected component $\widetilde C$ of $\widetilde{\mathcal C}_{N,a}$, one has $[\Delta_3]\in \widetilde C$. Since $\widetilde{\mathcal C}_{N,a}$ is compact in
$(\widetilde{\mathcal C}_N,\rho_{BM})$, the component $\widetilde C$ is also compact. Choose
\begin{equation}\label{min22}
  [Q]\in \widetilde C,\qquad {\rm such~that}\qquad
  \cP(Q)=\min_{[P]\in \widetilde C}\cP(P).
\end{equation}

\noindent{\it Claim.} Neither $Q$ nor its Santal\'o polar $L:=Q^{s(Q)}$ admits a non-trivial shadow flow. 
\vskip 2pt
Indeed, suppose first that $Q$ admits a
non-trivial shadow flow $(Q_t)$. 
By Lemma \ref{lem:facet-persistence}, there exists $c>0$ such that shadow flow $(Q_t)_{t\in[-c,c]}$
preserves the face lattice. 
It follows that $Q_t\in\mathcal C_N$, and $t\mapsto [Q_t]$ is a continuous path in $\widetilde{\mathcal C}_N$.  
By Lemma \ref{lem:strict-d-flow}, we may assume without loss of generality that $ \VP(Q_{t})<\VP(Q)$ for all $t\in (0,c]$. This is a contradiction, since
$Q_c \in \widetilde C $, but \[\VP(Q_c) < \VP(Q) = \min_{[P]\in \widetilde C}\cP(P). \]

It remains to exclude a non-trivial shadow flow of its Santal\'o polar $L=Q^{s(Q)}$. First, we connect $[Q]$ to $[L^{s(L)}]$ inside $\widetilde{\mathcal C}_{N,a}$. Note that $s(Q) = c(L)$.
Let
\[
  z_\tau=(1-\tau)s(Q)+\tau s(L),\qquad 0\le \tau\le 1.
\]
By Lemma~\ref{sliding lemma}, applied to $K=L$ and $z_0=s(Q)=c(L)$, we have
\[
  \cP(L^{z_\tau})
  \le
  \cP(L^{s(Q)})
  =
  \cP(Q),
  \qquad 0\le \tau\le 1 .
\]
Moreover, polarity reverses the face lattice, and hence $L^{z_\tau}\in\mathcal C_N$.
Thus $[L^{z_\tau}]$ is a path in $\widetilde{\mathcal C}_{N,a}$ starting from $[L^{z_0}] = [Q]$. In particular, $ [L^{s(L)}]\in \widetilde C.$ Since $[Q]$ minimizes $\VP$ on $\widetilde C$,  Lemma~\ref{lem:polar-nonincreasing} gives 
\begin{equation}\label{min33}
  \cP(L^{s(L)})=\cP(L)=\cP(Q).
\end{equation}

Now suppose that $L$ admits a non-trivial shadow flow $L_t$. By Lemma \ref{lem:facet-persistence}, there exists $c>0$ such that shadow flow $(L_t)_{t\in[-c,c]}$
preserves the face lattice of $L$.   
Hence, 
\[
  M_t:=L_t^{s(L_t)}\in\mathcal C_N.
\]
By Lemma \ref{lem:strict-d-flow}, we may assume without loss of generality that 
\[\VP(L_{t})<\VP(L),\qquad \forall t\in (0,c].\] 
By Lemma~\ref{lem:polar-nonincreasing} and \eqref{min33},
\[
  \cP(M_t) 
  \le
  \cP(L_t)
 < 
  \cP(L)
  =
  \cP(Q),
  \qquad 0 < t\le c .
\]
Therefore $[M_t]\in \widetilde{\mathcal C}_{N,a}$. This contradicts the assumption that $[Q]$ minimizes $\cP$ on $\widetilde C$. 

Thus neither $Q$ nor $Q^{s(Q)}$ admits a non-trivial shadow flow. By Lemma~\ref{lem:combinatorial-alternative}, $Q$ must be a tetrahedron. Hence $[\Delta_3]\in \widetilde C$. Since every connected component of $\widetilde{\mathcal C}_{N,a}$ contains the
same point $[\Delta_3]$, the set $\widetilde{\mathcal C}_{N,a}$ is connected.
\end{proof}

Here we invoke the following very useful fact  proved by Kim--Reisner in \cite{KimReisner2011} that the simplex is a strict local minimum for $\mathcal{P}$.

\begin{theorem}\cite[Theorem~1]{KimReisner2011}\label{Kim-Reisner stability}
   There exists $\delta(n)>0$ such that the following holds. Let $P$ be a simplex in $\mathbb{R}^n$ and $K$ a convex body in $\mathbb{R}^n$ with
$d_{BM}(K,P)=1+\delta$ for $0<\delta<\delta(n)$. Then
\[
\mathcal{P}(K) \geq \mathcal{P}(P) + C\delta,
\]
\textit{where $C=C(n)>0$.}
\end{theorem}
Here, we use only the three-dimensional case, and denote $\delta_0=\delta(3)$ and $C_0=C(3)$.
The following proposition gives a uniform stability result  for polytopes whose volume product is close to the minimum.

\begin{proposition}\label{strict convergence}
Let $(P_j)_{j\ge1}$ be a sequence of three-dimensional convex polytopes such that
\begin{equation}\label{limit}
    \mathcal{P}(P_j)\longrightarrow \frac{64}{9},
\end{equation}
as $j\to \infty$. Then we have
\begin{equation}\label{approx1}
d_{\BM}(P_j,\Delta_3)\longrightarrow 1.
\end{equation}
\end{proposition}

\begin{proof}
Suppose, for contradiction, \eqref{approx1} fails. Then there exists $0<\rho<\delta_0$ such that for infinitely many $j$,
\begin{equation}\label{upnumber}
d_{\BM}(P_j,\Delta_3)>1+\rho.
\end{equation}
Passing to this subsequence and relabeling, we may assume that \eqref{upnumber} holds for all $j$. 
Let $N_j$ be the number of vertices of $P_j$ and set $a_j=\mathcal{P}(P_j)$. By Proposition \ref{prop:strict-sublevel}, $\widetilde{\mathcal C}_{N_j, a_j}$ is connected and
$P_j,\Delta_3\in \widetilde{\mathcal C}_{N_j,a_j}$. Combining this with \eqref{upnumber} and the continuity of $L\mapsto d_{\BM}(L,\Delta_3)$, there exists $L_j\in \widetilde{\mathcal C}_{N_j,a_j}$ such that $d_{\BM}(L_j,\Delta_3)=1+\rho$. By Theorem~\ref{Kim-Reisner stability}, we have
\begin{equation}\label{strict positive}
     \mathcal{P}(L_j)\ge \mathcal P({\Delta_3})+C_0\rho= \frac{64}{9}+C_0\rho.
\end{equation}
Then, by \eqref{strict positive} and the definition of $ \widetilde{\mathcal C}_{N_j,a_j}$, we have
\[
\mathcal{P}(P_j)=a_j\ge \mathcal{P}(L_j) \ge \frac{64}{9}+C_0 \rho.
\]
Letting $j\to\infty$, we obtain a  contradiction to \eqref{limit}. Therefore $d_{\BM}(P_j,\Delta_3)\to 1$.
\end{proof}

\begin{theorem}\label{eq-case-main}
Let $K\subset \R^3$ be a convex body. Then
\[
\cP(K)=\frac{64}{9}
\]
if and only if $K$ is a tetrahedron, i.e. $K\in [\Delta_3].$
\end{theorem}

\begin{proof}
Suppose $\cP(K)=64/9.$ Choose a sequence $(P_j)_{j\ge1}$ of three-dimensional convex polytopes such that $P_j\to K$ in the Hausdorff metric as $j\to \infty$. By continuity of the volume product, we have
\[
\cP(P_j)\longrightarrow \cP(K)=\frac{64}{9},
\]
as $j\to \infty$. By Proposition \ref{strict convergence},
\[
d_{\BM}(P_j,\Delta_3)\longrightarrow 1.
\]
Moreover, Hausdorff convergence of convex bodies implies
\[
d_{\BM}(P_j,K)\longrightarrow 1.
\]
The triangle inequality for $d_{\BM}$ gives
\[
d_{\BM}(K,\Delta_3)\le d_{\BM}(K,P_j)\,d_{\BM}(P_j,\Delta_3)\longrightarrow 1.
\]
Since $d_{\BM}\ge 1$, we have $d_{\BM}(K,\Delta_3)=1$. Hence $K$ is a tetrahedron.

The converse follows from the affine invariance of $\mathcal P$ and the fact $\mathcal P(\Delta_3)=64/9.$
\end{proof}

The {\bf Main Result} of the current paper follows by combining Theorem \ref{thm:main} with Theorem \ref{eq-case-main}, and by recalling that the centroid of a simplex coincides with its Santal\'o point.

We end this section with a short remark. Although, for simplicity, we stated all results in this section in dimension three, Lemmas \ref{sliding lemma} and \ref{lem:strict-d-flow}  clearly also hold in $\mathbb R^n$.

\section{A shadow-flow  proof of the symmetric form}\label{sec:final}

In this section, we apply the shadow-flow method developed for the general form to the symmetric setting and give a new geometric proof of the following.

\begin{theorem}[Iriyeh--Shibata \cite{IS2020}]
\label{thm:main1}
Among all origin-symmetric convex bodies $K \subset \mathbb R^3$, that is, $K=-K$, one has
\[|K|\, |K^\circ| \ge \frac{32}{3}.\]
Equality holds if and only if $K$ is an affine image of a cube or an octahedron. 
\end{theorem}

\subsection{Notation and basic facts}

We collect some notation  and basic  facts in the symmetric setting, parallel to those developed in the general setting.

First, by the same proof of Lemma \ref{lem:bounded-class} we have
\begin{lemma}\label{lem:bounded-class-s}
Fix $N\geq 2n$, and let $\mathcal C_N^{\p \rm s}$ be the class of all origin-symmetric full-dimensional convex polytopes in $\R^n$ with at most $N$ vertices. Then $\VP$ attains its infimum on $\mathcal C_N^{\p \rm s}$.
\end{lemma}

Let $P\subset \R^3$ be an origin-symmetric polytope, and let 
\[
        V=V(P),\qquad E=E(P),\qquad F=F(P)
\]
denote the numbers of vertices, edges, and facets. 

A vector
\[
\alpha=(\alpha_1,\ldots,\alpha_V)\in\mathbb R^V
\]
is called a \emph{symmetric  speed vector} if \(x_i=-x_j\) implies \(\alpha_i=-\alpha_j\). We say $\alpha \in \R^V$ is a  {\it  symmetric $\theta$-admissible speed}, if $\alpha$ is symmetric and belongs to $A_\theta(P)$. 
The linear space of all  symmetric $\theta$-admissible speeds is denoted by $A_\theta^{\rm s}(P)$. 
The space of {\it symmetric trivial speeds} is denoted by
\[\mathcal T^s(P) = \left \{(w\cdot x_1, \ldots , w\cdot x_V) \in \R^V : w\in\R^3 \right\}. \]
It is simple to see that $\mathcal T^s(P)\subset A_\theta^{\rm s}(P)$ and $\dim\mathcal T^s(P) = 3.$

Let
\(\theta\in \sn\) be a direction. A {\it symmetric shadow flow} along $\theta$ is a shadow system 
with a symmetric $\theta$-admissible speed
\(\alpha\in A_\theta^{\rm s}(P)\):
\[
        x_i(t)=x_i+t\alpha_i\theta,
        \qquad
        P_t=\conv\{x_i(t):1\leq i\leq V\}, \quad t\in [-c,c],
\]
where $c>0$. 

The persistence of the face lattice and the volume-affine property of shadow flows remain valid in the symmetric setting. Namely, Lemmas \ref{lem:facet-persistence} and \ref{lem:volume-affine} still hold for $P_t$.
We also note that Lemmas \ref{lem:local-exclusion} and \ref{lem:minimizer-exclusion} carry over to the symmetric form with the same proof. Consequently, we obtain the following symmetric analogue of Lemma \ref{lem:minimizer-exclusion}.

\begin{lemma} \label{lem:minimizer-exclusion-s}
Fix $N\geq6$, and let $Q\in\mathcal C_N^{\p \rm s}$ be a minimizer of $\VP$ over $\mathcal C_N^{\p \rm s}$. Then neither $Q$ nor   $Q^{\circ}$ admits a non-trivial symmetric shadow flow along any direction $\theta\in \sn$.
\end{lemma}

\subsection{Dimension estimate for symmetric admissible speeds} 

Let $G$ be a facet of $P$. Recall that we use $m(G)$ to denote the number of vertices of $G.$
Since $P$ is origin-symmetric, we have $m(G)=m(-G)$. For a direction $\theta\in\sn$, define
\begin{equation}\label{eq:Ctheta-definition}
        C_\theta(P)=
        \frac12 \cdot \sum_{G:\,\theta\in\lin(G-G)}(m(G)-3),
\end{equation}
where the sum is over all facets parallel to $\theta$. 

We shall derive the following dimension estimate for the linear space of symmetric admissible speeds, in terms of \(V(P)\), \(F(P)\), and \(C_\theta(P)\).

\begin{lemma}\label{lem:symmetric-dimension-count}
Let $P=-P\subset\R^3$ be an origin-symmetric polytope. Then
\begin{equation}\label{eq:dimension-estimate}
        \dim A_\theta^{\rm s}(P)
        \geq
        \frac{F(P)-V(P)}2+2+C_\theta(P).
\end{equation}
Consequently, if the right-hand side of \eqref{eq:dimension-estimate} is greater than $3$, then $P$ admits a non-trivial symmetric shadow flow.
\end{lemma}

The proof follows from the same argument as the proof of Lemma \ref{lem:dimension-count}. For the reader's convenience, we sketch it here.
\begin{proof}
Since the speeds
\(\alpha\in\R^V\) are symmetric (in fact, they are odd vectors), their values are determined by choosing one
vertex from each opposite pair. Thus the number of independent vertex
variables is \(V/2\).

Let
\[
        \mathcal F_\theta
        =
        \{G\in \mathcal F:\theta\notin \lin(G-G)\}
\]
be the set of facets not parallel to \(\theta\). For each facet \(G\in \mathcal F_\theta\), the number of linear restrictions imposed by \(G\) is at most \(m(G)-3\), since an affine function on the plane containing \(G\) is determined by its values at three non-collinear vertices. If instead
\[
        \theta\in \lin(G-G),
\]
then \(G\) imposes no constraint by the definition of admissibility. Since opposite facets give the same restrictions after imposing the symmetry of the speed, we have
\[
        \frac V2 - \dim A_{\theta}^{\rm s}(P)
        \leq
        \frac12\sum_{G\in \F_\theta} \bigl(m(G)-3\bigr).
\]
It follows that
\[
\begin{aligned}
        \dim A_{\theta}^{\rm s}(P)
        &\geq
        \frac V2
        -\frac12\sum_{G\in \F_\theta} \bigl(m(G)-3\bigr)  \\
        &=
        \frac V2+\frac32F-\frac12\sum_{G\in \F}m(G)+C_\theta(P) \\
        &=
        \frac V2+\frac32F-E+C_\theta(P).
\end{aligned}
\]
By Euler's formula, the last expression is equal to
\[
        \frac{F-V}{2}+2+C_\theta(P).
\]
Hence \eqref{eq:dimension-estimate} follows.

Note that the trivial speed space
\[\mathcal T^s(P) =   \left \{(w\cdot x_1, \ldots , w\cdot x_V) \in \R^V : w\in\R^3 \right\} \]
is a three-dimensional subspace of \(A_\theta^{\rm s}(P)\). Hence, if the lower bound in \eqref{eq:dimension-estimate} is greater than \(3\), then
\(A_\theta^{\rm s}(P)\) contains a symmetric admissible speed that is not globally affine. Therefore it defines a non-trivial symmetric admissible shadow
system
\[
        P_t=\conv\{x_i+t\alpha_i\theta:1\leq i\leq V\}
\]
for sufficiently small \(|t|\).
\end{proof}

The following lemma implies that a three-dimensional origin-symmetric convex polytope with only trivial symmetric shadow flows must be a parallelepiped or an affine image of an octahedron.

\begin{lemma}\label{lem:combinatorial}
Let $P=-P\subset\R^3$ be an origin-symmetric polytope.  If neither $P$ nor $P^\circ$ admits a non-trivial symmetric shadow flow, then $P$ is a parallelepiped or an affine image of an octahedron.
\end{lemma}

\begin{proof}
Since $C_\theta(P)\geq 0$ and $P$ admits no non-trivial symmetric shadow flow, Lemma \ref{lem:symmetric-dimension-count} gives
\[
        \frac{F(P)-V(P)}2+2\leq3,
\]
and hence
\begin{equation*}\label{eq:F-minus-V}
        F(P)-V(P)\leq2.
\end{equation*}
Applying the same argument to $P^\circ$, whose numbers of vertices and facets are interchanged by polarity, gives
\begin{equation*}\label{eq:V-minus-F}
        V(P)-F(P)\leq2.
\end{equation*}
Thus
\begin{equation}\label{eq:abs-difference}
        |V(P)-F(P)|\leq2.
\end{equation}
Since $P$ is origin symmetric, $V(P)$ and $F(P)$ are even. Hence only the following three cases may occur.

\smallskip
\noindent\emph{Case 1: $V(P)=F(P)+2$.}
 Then we have
\[
        \frac{F(P^\circ)-V(P^\circ)}2
        =\frac{V(P)-F(P)}2=1.
\]
If some vertex $v\in \mathcal{V}(P)$ had degree $d(v)>3$, then the dual facet $G_v$ of $P^\circ$ would have $m(G_v)=d(v)>3$ vertices by Lemma \ref{lem:polarity-face-lattice}.  Choosing $\theta\in\lin(G_v-G_v)$, we would get
\[
        C_\theta(P^\circ)\geq d(v)-3\geq1.
\]
By Lemma \ref{lem:symmetric-dimension-count},
\[
        \dim A_\theta^{\rm s}(P^\circ)
        \geq1+2+1=4>3,
\]
contradicting the assumption on $P^\circ$.  Therefore every vertex of $P$ has degree $3$, which implies that
\[
        2E=3V.
\]
Using Euler's formula and $F=V-2$, we have that
 $V=8$ and $F=6$.

An origin-symmetric three-dimensional polytope with six facets has three opposite facet pairs, so it must be a parallelepiped after performing an affine transform.

\smallskip
\noindent\emph{Case 2: $F(P)=V(P)+2$.}
Then
\[
V(P^\circ)=F(P)=V(P)+2=F(P^\circ)+2.
\]
Moreover, the assumption is invariant under polarity: neither \(P^\circ\) nor \((P^\circ)^\circ=P\) admits a non-trivial symmetric shadow flow.
Thus Case 1 applied to \(P^\circ\) shows that \(P^\circ\) is a parallelepiped.
Consequently \(P\) is the polar of a parallelepiped, hence an affine image of the octahedron.

\smallskip
\noindent\emph{Case 3: $V(P)=F(P)$.}
Here Lemma \ref{lem:symmetric-dimension-count} becomes
\[
        \dim A_\theta^{\rm s}(P)\geq2+C_\theta(P).
\]
We claim that every facet is a triangle or a quadrilateral. Indeed, if $m(G)\geq 5$ for some facet $G$ of $P$, choosing $\theta\in\lin(G-G)$ would give $C_\theta(P)\geq2$, hence a non-trivial speed.

We also claim that no two quadrilateral facets share an edge. Indeed, if two quadrilateral facets shared an edge, then they would not be opposite facets; choosing \(\theta\) parallel to this edge would give  $C_\theta(P)\geq2$, again impossible.

Let $p$ be the number of triangular facets and $q$ the number of quadrilateral facets. Since $p+q=F=V$ and $3p+4q=2E,$ using Euler's formula we have that
\[
        p=4,
        \qquad
       q=V-4.
\]
Since no two quadrilateral facets share an edge, every edge of a quadrilateral is adjacent to a triangular facet. By counting the quadrilateral--triangle adjacency edges, we obtain
\[
        4q\leq3p=12.
\]
Since the number $q$ is even, we have that $q\leq2$ and therefore $V=q+4\leq6$. However, an origin-symmetric full-dimensional polytope in $\R^3$ has at least three opposite vertex pairs. Hence $V=6$ and $F=6.$ This cannot happen since an origin-symmetric three-dimensional polytope with six vertices must be an affine image of the octahedron and has eight facets, contradicting $F=6$. Thus Case 3 cannot occur.

In conclusion, the only remaining possibilities are the parallelepiped and the affine image of an octahedron.
\end{proof}

\subsection{New proof of Theorem \ref{thm:main1} }\label{sec:main-theorem-s}

\begin{proof}[Proof of the inequality]
We first prove the result for origin-symmetric polytopes. Fix \(N\geq 6\). By Lemma~\ref{lem:bounded-class-s}, there exists \(Q\in\mathcal C_N^{\p \rm s}\)
minimizing \(\VP\) on \(\mathcal C_N^{\p \rm s}\). By
Lemmas \ref{lem:minimizer-exclusion-s} and \ref{lem:combinatorial}, \(Q\) is a parallelepiped or an affine  image of the octahedron.

Let \(Q_3=[-1,1]^3\) be the cube. Since \(\VP\) is affine invariant and \(Q_3^\circ\) is the octahedron, a direct computation gives
\[
        \VP(Q_3)= \VP(Q_3^\circ)=\frac{32}{3}.
\]
Thus the minimum of \(\VP\) on \(\mathcal C_N^{\p \rm s}\) is \(32/3\). Since \(N\) is arbitrary, the lower bound holds for every origin-symmetric three-dimensional polytope.

Now let \(K=-K\subset\R^3\) be an arbitrary origin-symmetric convex body.
Choose a sequence \((P_m)_{m\geq 1}\) of origin-symmetric three-dimensional polytopes converging to \(K\) in the Hausdorff metric. By the polytope case and
the Hausdorff-continuity of the volume product,
\[
        \VP(K)=\lim_{m\to\infty}\VP(P_m)\geq\frac{32}{3}.
\]
\end{proof}

\begin{proof}[Characterization of the equality cases]
    The equality case was already characterized in \cite{IS2020}; see also \cite{FHMRZ} for another proof. We include here only a brief sketch of an alternative proof, since the argument is parallel to the equality argument in Section~\ref{sec:main6.2}; thus we only indicate the necessary modifications.

    The main difference is that, in the symmetric setting, there are two local-minimizer classes in three dimensions, namely $Q_3$ and $Q_3^\circ.$ Therefore, one cannot expect that the entire sublevel set $\widetilde{\mathcal C}_{N,a}^{s}=\{[P]\in \mathcal C^s_N/\Aff(3):\mathcal P(P)\le a\}$ is connected for \(N\ge 6\) and \(a\ge \frac{32}{3}\). However, by Lemma \ref{lem:combinatorial}, one can prove that each connected component contains a member of $\mathcal H:=\{[Q_3],[Q_3^\circ]\}.$ For an origin-symmetric convex body \(K\), write
\[
        d(K,\mathcal H)=\min\{d_{BM}(K,Q_3),d_{BM}(K,Q_3^\circ)\}.
\]
We  invoke the following local stability consequence of Kim's strict local minimality theorem for Hanner polytopes. For convenience, we state it only in dimension three.
\vskip 5pt
\noindent {\bf Theorem of Kim. \cite{Kim2014}} \label{hanner}
    {\it Let \(K\subset \mathbb R^3\) be an origin-symmetric convex body sufficiently close to one of the Hanner polytopes, in the sense that $\delta= d(K,\mathcal H)-1 $ is sufficiently small. Then
\[
\mathcal P(K) \geq \mathcal P(Q_3) +C\,\delta,
\]
where \(C>0\) is a constant depending only on the dimension.}
\vskip 5pt

By an argument similar to the proof of Proposition \ref{strict convergence}, we  obtain a uniform stability result for origin-symmetric  polytopes $P_j\subset \R^3$ whose volume product is close to the minimum: If $
\mathcal P(P_j)\rightarrow \frac{32}{3},$ then we have $d(P_j,\mathcal H)\rightarrow 1$ as $j\rightarrow \infty.$

Finally, let  $K \subset\mathbb R^3$  be an origin-symmetric convex body with $\mathcal P(K)=32/3.$ By the same argument as in the proof of Theorem~\ref{eq-case-main}, we obtain $d(K,\mathcal H)=1$. Thus $K$ is an affine image of either $Q_3$ or $Q_3^\circ$.
\end{proof}

\section*{Acknowledgements}

The authors would like to thank Prof. Shlomo Reisner for valuable comments and suggestions on an earlier version of this paper.

\end{document}